\documentclass[11pt]{article}
\usepackage{HongLiu}

\title{Polynomial minimum degree stability for $3$-chromatic graphs}
\author{
Yisai Xue\thanks{School of Mathematics and Statistics, Ningbo University, Ningbo, China. Email: \texttt{xueyisai@nbu.edu.cn}. Supported by the National Natural Science Foundation of China (No. 12501486).}
}
\date{}
\begin{document}
\maketitle

\begin{abstract}
Let $H$ be a fixed 3-chromatic graph, and let $g \geq 2$ be the smallest integer such that there is no homomorphism $H \rightarrow C_{2 g+1}$. For every $\varepsilon>0$, we prove that there exists a constant $\rho>0$ such that every $H$-free $n$-vertex graph $G$ with $\delta(G) \geq(2 /(2 g+1)+\varepsilon) n$ can be made bipartite by deleting $O\left(n^{2-\rho}\right)$ edges. Thus the sharp qualitative minimum-degree stability theorem for 3-chromatic graphs admits a polynomial strengthening. In particular, this gives an affirmative answer to a question of Illingworth [\textit{Minimum degree stability of $H$-free graphs}, Combinatorica, 43(1):129-147, 2023.] on blow-ups of odd cycles.
\end{abstract}

\section{Introduction}

A central theme in extremal graph theory is to understand the structure of graphs which avoid a fixed subgraph.  The classical Erd\H{o}s-Stone-Simonovits Theorem \cite{ErdosStone1946,ErdosSimonovits1966} says that if $H$ has chromatic number $r+1$, then every $n$-vertex $H$-free graph has at most $\big(1-\frac{1}{r}+o(1)\big)\binom{n}{2}$ edges, and the corresponding stability theorem says that every nearly extremal $H$-free graph is close to being $r$-partite.

A minimum-degree version of this problem goes back to Andr\'asfai, Erd\H{o}s and S\'os \cite{AES74}, who proved that every $K_{r+1}$-free graph $G$ on $n$ vertices with $\delta(G)>\frac{3r-4}{3r-1}n$ is $r$-partite.  Alon and Sudakov \cite{alon2006h} extended this phenomenon to arbitrary fixed forbidden graphs: if $\chi(H)=r+1$, then every sufficiently large $H$-free graph with minimum degree at least $\big(\frac{3r-4}{3r-1}+\varepsilon\big)n$ can be made $r$-partite by deleting $O(n^{2-\rho})$ edges, for some $\rho=\rho(H)>0$.  Allen \cite{allen2010dense} later gave a regularity-free proof and obtained sharp estimates, up to a constant factor, in terms of a Zarankiewicz-type extremal function.

Recently, Illingworth \cite{illingworth2023minimum} initiated a systematic study of minimum-degree stability, asking for the minimum degree threshold above which every $H$-free graph must be close to $r$-partite.
Unlike the classical Tur\'an density, this threshold is not determined solely by the chromatic number of $H$. 
For 3-chromatic graphs, {\L}uczak and Simonovits \cite{luczak2008minimum} had earlier determined the corresponding qualitative bipartite-stability threshold in terms of an equivalent odd-cycle homomorphism parameter.
In Illingworth’s formulation, if $g$ is the smallest positive integer such that there is no homomorphism $H\to C_{2g+1}$, then the threshold is $2/(2g+1)$.

For $\ell\ge 3$ and $t\ge 1$, let $C_\ell[t]$ denote the \textit{$t$-blow-up} of $C_\ell$, that is, the graph obtained by replacing each vertex of $C_\ell$ by an independent set of size $t$ and each edge by a complete bipartite graph.
Illingworth \cite{illingworth2023minimum} showed that for every $g\ge 2$, $t\ge 1$ and $\varepsilon>0$, every sufficiently large $C_{2g-1}[t]$-free $n$-vertex graph $G$ with
$\delta(G)\ge \big(\frac{2}{2g+1}+\varepsilon\big)n$
can be made bipartite by deleting $o(n^2)$ edges. Furthermore, Illingworth asked whether this conclusion can be strengthened to a polynomial bound.

\begin{ques}[Question 16 in \cite{illingworth2023minimum}]
    For positive integers $g$ and $t$, and every $\varepsilon>0$, is there some $\rho>0$ such that every $n$-vertex graph with minimum degree at least $(2 /(2 g+1)+\varepsilon) n$ either contains $C_{2 g-1}[t]$ or can be made bipartite by deleting $O(n^{2-\rho})$ edges?
\end{ques}

We answer this question affirmatively and, in fact, obtain the polynomial strengthening for
every fixed 3-chromatic forbidden graph. For a graph $G$, define
$$
\gamma_2(G)=\min_{V(G)=X\cup Y}\bigl(e(G[X])+e(G[Y])\bigr).
$$
Thus $\gamma_2(G)$ is the minimum number of edges that must be deleted from $G$ to make it bipartite.  
For fixed $g\ge2$, write $\theta_g=\frac{2}{2g+1}$.

\begin{theorem}\label{thm:3-chromatic}
    Let $H$ be a fixed graph with $\chi(H)=3$, and let $g \geq 2$ be the smallest integer such that $H \nrightarrow C_{2 g+1}$. For every $\varepsilon>0$, there exist constants $C>0, \rho>0$, and $n_0$ such that every $H$-free graph $G$ on $n \geq n_0$ vertices with
$\delta(G) \geq\left(\theta_g+\varepsilon\right) n$
satisfies $\gamma_2(G) \leq C n^{2-\rho}$.
\end{theorem}

We shall deduce Theorem \ref{thm:3-chromatic} from the following special case, which is the form proved in the remainder of the paper.

\begin{theorem}\label{thm:main}
For every $g\ge 2$, $t\ge 1$, and $\varepsilon>0$, there exist constants $C>0$, $\rho>0$, and $n_0$ such that the following holds for every $n\ge n_0$.  Let $G$ be an $n$-vertex graph satisfying
$$
\delta(G)\ge \left(\theta_g+\varepsilon\right)n.
$$
Then either $G$ contains a copy of $C_{2g-1}[t]$, or $\gamma_2(G)\le Cn^{2-\rho}$.
\end{theorem}

\begin{proof}[Proof of Theorem \ref{thm:3-chromatic} assuming Theorem \ref{thm:main}.]
     By the minimality of $g$, there is a homomorphism $\varphi: H \rightarrow C_{2 g-1}$. Since $H$ is fixed, it is a subgraph of $C_{2 g-1}[t]$ for some fixed $t$.
     Thus every $H$-free graph is $C_{2 g-1}[t]$-free, and Theorem \ref{thm:main} gives the result.
\end{proof}

The threshold in Theorem~\ref{thm:3-chromatic} is best possible. Indeed, the balanced blow-up of $C_{2 g+1}$ is $H$-free, since otherwise the natural homomorphism from this blow-up to $C_{2 g+1}$ would yield $H \rightarrow C_{2 g+1}$. It has minimum degree asymptotic to $\theta_g n$ and distance $\Omega(n^2)$ from being bipartite. Thus the coefficient $2 /(2 g+1)$ cannot be decreased.

\paragraph{Proof overview.}
It remains to prove Theorem \ref{thm:main}.
Let $\ell=2g-1$ and write $\gamma_2(G)=\mu n^2$.  We first show that, unless $\mu$ is already polynomially small, a random induced subgraph $G[S]$ on $q=\Omega(\mu^{-5})$ vertices inherits both a positive proportion of the distance from bipartiteness and the minimum-degree condition.  A reduction to the $2$-connected non-bipartite case, together with H\"aggkvist's theorem, then implies that such a random set $S$ contains a copy of $C_\ell$ with positive probability.  It follows that $G$ contains many copies of $C_\ell$.

On the other hand, if $G$ is $C_\ell[t]$-free, then the theorem of Alon and Shikhelman implies the number of copies of $C_\ell$ is at most $n^{\ell-\alpha}$ for some $\alpha>0$.  Comparing this upper bound with the lower bound obtained from random sampling forces $\gamma_2(G)=O(n^{2-\rho})$, as required.

\paragraph{Notation.}
All graphs in this paper are finite and simple.  For a graph $G$, we write
$V(G)$ and $E(G)$ for its vertex set and edge set, respectively, and let $v(G)=|V(G)|$ and $e(G)=|E(G)|$.  
For a vertex $x\in V(G)$, let $d_G(x)$ denote the degree of $x$ in $G$, and let $\delta(G)$ denote the minimum degree of $G$.
For an integer $\ell$, let $C_\ell$ be the cycle of length $\ell$.
For $X\subseteq V(G)$, let $G[X]$ be the subgraph induced by $X$. For $X,Y\subseteq V(G)$, let $e_G(X,Y)$ denote the number of edges
of $G$ with one endpoint in $X$ and the other in $Y$.

\section{Reduction to the two-connected non-bipartite case}

In this section we reduce the use of H\"aggkvist's theorem to the
2-connected non-bipartite case.

\begin{lemma}[H\"aggkvist, \cite{haggkvist1982odd}]\label{lem:haggkvist}
    Let $G$ be an $n$-vertex $2$-connected non-bipartite graph with
$\delta(G) > \lfloor \frac{2n}{2g+1} \rfloor$
and
$n > \binom{g+1}{2}(2g+1)(3g-1)$.
Then $G$ contains a copy of $C_{2g-1}$.
\end{lemma}

\begin{lemma}
\label{lem:two-connected-reduction}
Fix $g\ge 2$ and $\eta>0$.
Let $N_0(g,\eta)=\max\big\{
  \frac{g+1}{\eta},
  \binom{g+1}{2}(2g+1)^2(3g-1)
  \big\}$
and $G$ be an $n$-vertex graph with $n>N_0$. Suppose that
\[
  \delta(G)\ge (\theta_g+\eta)n \quad\text{and}\quad \gamma_2(G)>(g+1)n.
\]
Then $G$ contains a copy of $C_{2g-1}$.
\end{lemma}

\begin{proof}
By the definition of $N_0$, for every $n> N_0$, we have
\[
  (\theta_g+\eta)n-g\ge \theta_g n,
  \qquad
  \text{and}
  \qquad
  \frac{n}{2g+1}> \binom{g+1}{2}(2g+1)(3g-1).
\]
We decompose $G$ as follows. Start with the connected components of $G$,
and assign each of them depth $0$. A component is declared terminal if it is
bipartite, or if it is non-bipartite and $2$-connected. 
If a current component $H$ is not terminal, then $H$ is connected, non-bipartite, and not $2$-connected. We choose a cutvertex $v$ of $H$, delete all edges of $H$ incident with $v$, and continue with the components of $H-v$, each assigned depth one larger. 
We next prove that the splitting process has bounded depth.

\begin{claim}\label{cl:depth}
    Every component appearing in the process has depth at most $g$.
\end{claim}

\begin{poc}
    Suppose, to the contrary, that some branch contains $g+1$ splitting steps. Consider the parent components at these $g+1$ splitting steps. 
    A vertex contained in a component at depth $d$ can be adjacent to at most one cutvertex removed at each of the $d$ previous splitting steps. Hence it loses at most $d$ incident edges, and so $\delta(H)\ge \delta(G)-d$.
    Therefore, for each such parent component $H$, 
$$
\delta(H)\ge \delta(G)-g\ge (\theta_g+\eta)n-g>\theta_g n.
$$
At each of these $g+1$ splitting steps, there is a sibling component not lying on the chosen branch. Let $Q$ be such a sibling component, arising when a parent component $H$ is split at a cutvertex $v$. For any $x\in Q$, all neighbours of $x$ in $H$ lie either in $Q$ or are equal to $v$. Hence $|Q|\ge d_H(x)\ge \delta(H)>\theta_g n$. 

The sibling components obtained from the $g+1$ different splitting steps are pairwise disjoint. Hence $G$ contains more than
$(g+1)\theta_g n=\frac{2g+2}{2g+1}n$ vertices, a contradiction. This proves the claim.
\end{poc}

By Claim \ref{cl:depth}, every terminal component has depth at most $g$. Hence if $H$ is terminal, then $\delta(H)\ge \delta(G)-g>\theta_g n$. In particular, $|H|>\theta_g n$. Since the terminal components are pairwise disjoint, there are at most $\lfloor (2g+1)/2\rfloor= g$ of them.

Each splitting step replaces one current component by at least two smaller components.
Therefore the number of current components increases by at least one at each step. 
Since the process ends with at most $g$ terminal components, the total number of splitting steps is at most $g$.

At each splitting step we delete at most $n$ edges. Hence the total number of
deleted edges is at most $gn$.
If every terminal component were bipartite, then after deleting these edges the
graph $G$ would become a disjoint union of bipartite graphs. Therefore
$\gamma_2(G)\le gn < (g+1)n$,
a contradiction. Thus some terminal component
$B$ is non-bipartite. By construction, $B$ is $2$-connected.
Since $B$ has depth at most $g$, we have
$$
\delta(B)\ge \delta(G)-g
\ge (\theta_g+\eta)n-g
> \theta_g |B|.
$$
Moreover, $|B|\ge \theta_g n> \binom{g+1}{2}(2g+1)(3g-1)$. Hence
Lemma~\ref{lem:haggkvist} applies to $B$, and yields a copy of
$C_{2g-1}$ in $B$, and therefore in $G$.
\end{proof}

\section{Many short odd cycles from large $\gamma_2$}

In this section we combine the sampling lemma with the reduction lemma to show that, above the minimum-degree threshold, a graph with large $\gamma_2$ must contain many copies of $C_{2g-1}$.  For a positive integer $\ell$, let $N_\ell(G)$ denote the number of copies of $C_\ell$ in $G$.  

% We use the following sampling lemma of Gishboliner and Shapira ~\cite{GishbolinerShapira2020}. 
% As usual, integer rounding in the sample size is suppressed.
Since $\gamma_2(G)=e(G)-\operatorname{MaxCut}(G)$,
the usual notion that $G$ is $\nu$-far from bipartiteness is exactly
the condition $\gamma_2(G)\ge \nu v(G)^2$.
Thus Lemma 4.1 of Gishboliner--Shapira~\cite{GishbolinerShapira2020} can be stated as follows.

\begin{lemma}[Gishboliner--Shapira~\cite{GishbolinerShapira2020}]\label{Gishboliner-Shapira}
    There exists an absolute constant $c>0$ such that the following holds.
    Let $\nu\in(0,1)$, and let $H$ be an $n$-vertex graph satisfying $\gamma_2(H)\ge \nu n^2$,
    where $n\ge c\nu^{-5}$.
    Let $q=\lceil c\nu^{-5}\rceil$, and let $Q$ be chosen uniformly at random from
    $\binom{V(H)}{q}$. Then
    \[
        \Pr\left(
            \gamma_2(H[Q])\ge \frac{\nu}{2}q^2
        \right)
        \ge \frac{2}{3}.
    \]
\end{lemma}

\begin{lemma}[Hypergeometric Chernoff bound, \cite{JLR2000}]
\label{lem:chernoff}
For every $\xi>0$, there exists a constant $c=c(\xi)>0$ such that the
following holds.  Let $X$ be a hypergeometric random variable arising from a
sample of size $q$, and suppose that
$\mathbb E X\ge aq$
for some $a>0$.  Then
\[
  \Pr\left(X\le (a-\xi)q\right)\le e^{-cq}.
\]
\end{lemma}

The next lemma is the main consequence of the sampling and reduction steps.

\begin{lemma}
\label{lem:many-specified-odd-cycles}
Fix $g\ge 2$ and $\varepsilon>0$, and let $\ell=2g-1$.
There exist constants $C_0=C_0(g,\varepsilon)>0$ and $A=A(g,\varepsilon)>0$ such that the following holds. Let $G$ be an $n$-vertex graph
satisfying
$\delta(G)\ge (\theta_g+\varepsilon)n$
and let $m=\gamma_2(G)$. Then either
$$
m\le C_0 n^{2-1/5} \qquad\text{or}\qquad N_\ell(G)\ge A
\left(\frac{m}{n^2}\right)^{5\ell} n^\ell.
$$
\end{lemma}

\begin{proof}
If $m=0$, then the first alternative holds.  We may therefore assume that $m>0$.  Write $\mu=m/n^2$, since $\gamma_2(G)\le e(G)\le n^2/2$, we have $0<\mu\le 1/2$.

Let $C_1>0$ be the absolute constant given by Lemma \ref{Gishboliner-Shapira}, and let $C_2=C_2(\varepsilon / 4)>0$ be the constant given by Lemma \ref{lem:chernoff}. Choose $a \in(0,1]$ sufficiently small, and set
 $K=C_1a^{-5}$, so that
\begin{equation}\label{eq:L42-K-choices}
\begin{gathered}
K>\max\left\{N_0(g,\varepsilon/2),2(g+1),\frac{4(\theta_g+3\varepsilon/4)}{\varepsilon},\frac{g+1}{8a}\right\}\quad\text{and}\quad
xe^{-C_2(x-1)}\le \frac{1}{6} \quad\text{for every }x\ge K.
\end{gathered}
\end{equation}
Set $$C_0=2K^{1/5}
\quad\text{and}\quad
A=\frac{1}{2(2K)^\ell}.$$ 
Let $q=\lceil K\mu^{-5}\rceil$.
We first dispose of the case $q>n$.  In this case $K\mu^{-5}>n-1$, and hence
$$
m=\mu n^2<\left(\frac{K}{n-1}\right)^{1/5}n^2\le 2K^{1/5}n^{2-1/5}.
$$
Thus the first alternative holds.  Henceforth assume that $q\le n$.

Let $S\subseteq V(G)$ be a uniformly random $q$-set, and let $F=G[S]$.  
Since $\gamma_2(G)=\mu n^2 \geq a \mu n^2$ and $K \mu^{-5}=C_1(a \mu)^{-5}$,
Lemma~\ref{Gishboliner-Shapira}, applied with $\nu=a\mu$, gives
\begin{equation}\label{eq:L42-E1}
\mathbb P\left(\gamma_2(F)\ge a\mu q^2/2\right)\ge 2/3.
\end{equation}
Let $\mathcal E_1$ denote the event $\gamma_2(F)\ge a\mu q^2/2$.

Next we show that $F$ inherits the minimum-degree condition with high probability.  Fix $v\in V(G)$ and condition on $v\in S$.  Then $d_F(v)=|N_G(v)\cap(S\setminus{v})|$ is hypergeometric with sample size $q-1$, and
$$
\mathbb E[d_F(v)\mid v\in S]
=d_G(v)\frac{q-1}{n-1}
\ge (\theta_g+\varepsilon)(q-1).
$$
Since $q\ge K$, the choice of $K$ in \eqref{eq:L42-K-choices} gives
\begin{align*}
 (\theta_g+\varepsilon/2)q
\le
(\theta_g+3\varepsilon/4)(q-1).   
\end{align*}
Therefore Lemma~\ref{lem:chernoff}, applied with sample size $q-1$ and $\xi=\varepsilon/4$, gives
$$
\mathbb P\left(d_F(v)\le (\theta_g+\varepsilon/2)q\mid v\in S\right)\le e^{-C_2(q-1)}.
$$

Let $X$ be the number of vertices $v\in S$ with $d_F(v)\le(\theta_g+\varepsilon/2)q$.  
Then $\mathbb E X\le qe^{-C_2(q-1)}\le 1/6$ by \eqref{eq:L42-K-choices}, and Markov's inequality gives
\begin{equation}\label{eq:L42-E2}
\mathbb P\left(\delta(F)\ge(\theta_g+\varepsilon/2)q\right)\ge \frac56.
\end{equation}
Let $\mathcal E_2$ denote the event in \eqref{eq:L42-E2}.  Combining \eqref{eq:L42-E1} and \eqref{eq:L42-E2}, we obtain
\begin{equation}\label{eq:L42-good-events}
\mathbb P(\mathcal E_1\cap\mathcal E_2)\ge \frac23+\frac56-1=\frac12.
\end{equation}

Assume that $\mathcal E_1\cap\mathcal E_2$ holds.  Then $\delta(F)\ge(\theta_g+\varepsilon/2)q$ and $\gamma_2(F)\ge a\mu q^2/2$.  Moreover, by \eqref{eq:L42-K-choices}, we have
\begin{align*}
\gamma_2(F) \geq \frac{a \mu q}{2} q \geq \frac{a K}{2} \mu^{-4} q \geq 8 a K q>(g+1) q .
\end{align*}
 Also $q\ge K>N_0(g,\varepsilon/2)$.  Hence Lemma~\ref{lem:two-connected-reduction}, applied to $F$ with $\eta=\varepsilon/2$, implies that $F$ contains a copy of $C_{2g-1}=C_\ell$.  Together with \eqref{eq:L42-good-events}, this yields
 \begin{equation}\label{eq:L42-cycle-lower}
 \mathbb P(G[S]\text{ contains a copy of }C_\ell)\ge \frac{1}{2}.
 \end{equation}

It remains to convert this probability lower bound into a lower bound on $N_\ell(G)$.  A fixed copy of $C_\ell$ in $G$ is contained in $S$ with probability $\binom{n-\ell}{q-\ell}/\binom nq\le \big(\frac qn\big)^\ell$.
Hence, by the union bound,
\begin{equation}\label{eq:L42-cycle-upper}
\mathbb P(G[S]\text{ contains a copy of }C_\ell)
\le
N_\ell(G)\left(\frac qn\right)^\ell.
\end{equation}
Combining \eqref{eq:L42-cycle-lower} and \eqref{eq:L42-cycle-upper}, we get $N_\ell(G)\ge \frac12\big(\frac{n}{q}\big)^\ell$.  Finally, since $q=\lceil K\mu^{-5}\rceil$ and $K\mu^{-5}\ge 1$, we have
$q\le 2K\mu^{-5}$. Hence
$$
N_\ell(G)
\ge
\frac12\left(\frac nq\right)^\ell
\ge
\frac{1}{2(2K)^\ell}\mu^{5\ell}n^\ell=
A
\left(\frac{m}{n^2}\right)^{5\ell}n^\ell.
$$
This proves the second alternative, and hence the lemma.
\end{proof}

\section{Proof of Theorem~\ref{thm:main}}

We now finish the proof by comparing the lower bound on short odd cycles obtained in the previous section with a general upper bound of Alon and Shikhelman.  This comparison shows that a $C_{2g-1}[t]$-free graph above the threshold cannot have large $\gamma_2$.

For two graphs $F$ and $H$, we define  ${\rm ex}(n, F, H)$
 as the maximum possible number of copies of $F$ in an $H$-free graph on $n$ vertices. Alon and Shikhelman~\cite{alon2016many} proved the following result.

\begin{theorem}[\cite{alon2016many}]\label{thm:AS16}
Let $F$ be a graph on $f$ vertices. If $H$ is a subgraph of a blow-up of $F$, then there exists some $\alpha:=\alpha(F,H)>0$ such that 
\[
{\rm ex}(n, F, H) \leq n^{f-\alpha}.
\] 
\end{theorem}

\begin{proof}[\textbf{Proof of Theorem~\ref{thm:main}}]
Fix $g\ge 2$, $t\ge 1$, $\varepsilon>0$ and let $\ell=2g-1$.
Let $G$ be an $n$-vertex graph satisfying
$
  \delta(G)\ge
  \big(\theta_g+\varepsilon\big)n.
$
Suppose that $G$ is $C_\ell[t]$-free.  We shall prove that
$
  \gamma_2(G)\le Cn^{2-\rho}
$
for some $C>0$ and $\rho>0$.

Let
$
  m=\gamma_2(G).
$
By Lemma~\ref{lem:many-specified-odd-cycles}, applied with this $g$ and
$\varepsilon$, there exist constants $A=A(g,\varepsilon)>0$ and
$C_0=C_0(g,\varepsilon)>0$ such that either
\[
  m\le C_0 n^{2-1/5} \qquad \text{or} \qquad N_\ell(G)\ge
  A\left(\frac{m}{n^2}\right)^{5\ell}n^\ell.
\]
In the first case, the desired conclusion holds.
We may therefore assume that
\begin{align}\label{eq:lower}
      N_\ell(G)\ge
  A\left(\frac{m}{n^2}\right)^{5\ell}n^\ell.
\end{align}
We apply Theorem~\ref{thm:AS16} with $F=C_\ell$ and $H=C_\ell[t]$.
Hence there exists $\alpha=\alpha(\ell,t)>0$ such that
\begin{align}\label{eq:upper}
    N_\ell(G)\le n^{\ell-\alpha}.
\end{align}
Combining \eqref{eq:lower} and \eqref{eq:upper}, we get
\[
  A\left(\frac{m}{n^2}\right)^{5\ell}n^\ell
  \le
  n^{\ell-\alpha}.
\]
Hence
$\left(\frac{m}{n^2}\right)^{5\ell}\le A^{-1}n^{-\alpha}$.
It follows that
\[
  m
  \le
  A^{-1/(5\ell)}
  n^{2-\alpha/(5\ell)}.
\]
Now define $C=\max\{C_0,A^{-1/(5\ell)}\}$ and
$\rho=\min\left\{
    \frac1{5},
    \frac{\alpha}{5\ell}
  \right\}$.
Then in all cases,
$\gamma_2(G)=m\le Cn^{2-\rho}$.
This proves the theorem.
\end{proof}

\section{Concluding remarks}\label{sec:concluding}

\paragraph{Optimal exponent.}
Allen~\cite{allen2010dense} proved that if $\chi(H)=r+1$ and $\delta(G)>\big(\frac{3r-4}{3r-1}+\varepsilon\big)n$, 
then every $H$-free $n$-vertex graph can be made $r$-partite by deleting at
most $C\cdot\operatorname{biex}(n,H)$ edges, where $\operatorname{biex}(n,H)$ is
the Zarankiewicz-type extremal function coming from the bipartite graphs
induced by two colour classes of $H$.  In particular, his result determines
the correct order of the deletion distance up to a constant factor.  It is
natural to ask for an analogous description in the setting of Theorem~\ref{thm:3-chromatic}.

\begin{ques}
Fix a $3$-chromatic graph $H$  and $\varepsilon>0$, and let $g\ge2$ be the smallest integer such that $H\nrightarrow C_{2g+1}$. 
Determine the order of magnitude of
\[
M_{H,\varepsilon} (n)=
\max\{\gamma_2(G): v(G)=n, \delta(G)\ge(\theta_g+\varepsilon)n,~  H \not\subseteq G\}.
\]
\end{ques}

\section*{AI declaration}
All mathematical content is due to the author. ChatGPT Plus was used for proof checking and improving writing.

\section*{Acknowledgements}

The author thanks Hong Liu for bringing the result of {\L}uczak and Simonovits \cite{luczak2008minimum} to his attention.
The author also thanks an anonymous referee for bringing Lemma~\ref{Gishboliner-Shapira} to his attention.

\bibliographystyle{abbrv}
\bibliography{refs.bib}

\end{document}